\documentclass[12pt,reqno]{article}

\usepackage[usenames]{color}
\usepackage{amssymb}
\usepackage{amsmath}
\usepackage{amsthm}
\usepackage{amsfonts}
\usepackage{amscd}
\usepackage{graphicx}

\usepackage[colorlinks=true,
linkcolor=webgreen,
filecolor=webbrown,
citecolor=webgreen]{hyperref}

\definecolor{webgreen}{rgb}{0,.5,0}
\definecolor{webbrown}{rgb}{.6,0,0}

\usepackage{color}
\usepackage{fullpage}
\usepackage{float}

\usepackage{graphics}
\usepackage{latexsym}
\usepackage{epsf}
\usepackage{accents}

\begin{document}

\theoremstyle{plain}
\newtheorem{theorem}{Theorem}
\newtheorem{corollary}[theorem]{Corollary}
\newtheorem{lemma}{Lemma}
\newtheorem*{example}{Examples}
\newtheorem*{remark}{Remark}

\begin{center}
\vskip 1cm{\LARGE\bf 
Fibonacci Identities Involving Reciprocals of Binomial Coefficients\\
}
\vskip 1cm
\large
Kunle Adegoke \\
Department of Physics and Engineering Physics\\
Obafemi Awolowo University\\
220005 Ile-Ife, Nigeria\\
\href{mailto:adegoke00@gmail.com}{\tt adegoke00@gmail.com} \\
\end{center}

\vskip .2 in

\noindent 2010 {\it Mathematics Subject Classification}:
Primary 11B39; Secondary 11B37.

\noindent \emph{Keywords: }
Fibonacci number, Lucas number, summation identity, series, binomial coefficient, Horadam sequence.

\begin{abstract}
\noindent We derive some Fibonacci and Lucas identities which contain inverse binomial coefficients. Extension of the results to the general Horadam sequence is possible, in some cases.
\end{abstract}

\section{Introduction}
The binomial coefficients are defined, for non-negative integers $i$ and $j$, by
\[
\binom ij=
\begin{cases}
\frac{{i!}}{{j!(i - j)!}}, & \text{$i \ge j$};\\
0, & \text{$i<j$}.
\end{cases}
\]
For $n$ a non-negative integer, the identity
\begin{equation}\label{eq.woluw2n}
\sum_{j = 0}^n {\frac{1}{\binom nj}}  = \frac{{n + 1}}{{2^{n} }}\sum_{j = 0}^{n} {\frac{{2^j }}{j + 1}} 
\end{equation}
is well-known (\cite{gould72,rocket81,trif00,sury04,sprugnoli06}).

Let $F_j$ and $L_j$ be the $j^{th}$ Fibonacci and Lucas numbers. 

Our first primary goal in this paper is to derive analogous identities to \eqref{eq.woluw2n}, involving the Fibonacci and Lucas numbers and inverse binomial coefficients, namely,
\begin{gather}
\sum_{j = 0}^n {\frac{{( - 1)^{rj} F_{2rj + s} }}{\binom nj}}  = \frac{{(n + 1)F_{rn + s} }}{{L_r^{n + 1} }}\sum_{j = 0}^n {\frac{{( - 1)^{rj} L_r^{j} L_{r(j + 1)} }}{{j + 1}}},\label{eq.t8a4tnh}\\
\nonumber\\
\sum_{j = 0}^n {\frac{{( - 1)^{rj} L_{2rj + s} }}{\binom nj}}  = \frac{{(n + 1)L_{rn + s} }}{{L_r^{n + 1} }}\sum_{j = 0}^n {\frac{{( - 1)^{rj} L_r^{j} L_{r(j + 1)} }}{{j + 1}}}\label{eq.ne62ilo} ,
\end{gather}
where $s$ and $r$ are any integers and $n$ is a non-negative integer.

Identity~\eqref{eq.ne62ilo} reduces to~\eqref{eq.woluw2n} when $r=0$.

For the derivation of finite sums, including \eqref{eq.t8a4tnh} and \eqref{eq.ne62ilo}, we will employ the following identity (Gould~\cite[identity $2.4$]{gould72}):
\begin{equation}\label{eq.ibqxgb5}
\sum_{j = 0}^n {\frac{{z^j }}{\binom nj}}  = (n + 1)\left( {\frac{z}{{1 + z}}} \right)^n \frac{1}{{1 + z}}\sum_{j = 0}^n {\frac{{(1 + z^{j + 1} )}}{{j + 1}}\left( {\frac{{1 + z}}{z}} \right)^j }, 
\end{equation}
valid for all $z\ne-1$ and any non-negative integer $n$.

Identity~\eqref{eq.ibqxgb5} also corresponds to setting $m=0$ in the more general identity obtained by Sury et.~al.~\cite[identity (2)]{sury04}.

Our second objective is to derive infinite summation identities involving the inverse binomial coefficients and the Fibonacci and Lucas numbers. 

Let $\alpha=(1 + \sqrt5)/2$, the golden ratio, and $\beta=(1 - \sqrt5)/2=-1/\alpha$.

Among other results, we will show that:
\[
\sum_{j = 1}^\infty  {\frac{{F_{2j - 1} }}{\binom{2j}j}}  = \frac{3}{5} + \frac{{4\pi }}{{25}}\sqrt {\frac{{\alpha ^5 }}{{\sqrt 5 }}},\quad\sum_{j = 1}^\infty  {\frac{{L_{2j - 1} }}{\binom{2j}j}}  = 1 + \frac{{4\pi }}{5}\sqrt {\frac{\alpha }{{\sqrt 5 }}},
\]

\[
\sum_{j = 1}^\infty  {\frac{{F_{2j} }}{{j\binom{2j}j}}}  = \frac{{2\pi }}{5}\sqrt {\frac{\alpha }{{\sqrt 5 }}} ,\quad\sum_{j = 1}^\infty  {\frac{{L_{2j} }}{{j\binom{2j}j}}}  = \frac{{2\pi }}{5}\sqrt {\frac{{\alpha ^5 }}{{\sqrt 5 }}},
\]

\[
\sum_{j = 1}^\infty  {\frac{{L_{2j + 3} }}{{j^2 \binom{2j}j}}}  = \frac{{2\pi ^2}}{25}\alpha^3\sqrt5,\quad\sum_{j = 1}^\infty  {\frac{{L_{2j - 3} }}{{j^2 \binom{2j}j}}}  = \frac{{2\pi ^2}}{25}\beta^3\sqrt5,
\]

\[
\sum_{j = 0}^\infty  {\frac{{2^{2j + 1} }}{{(2j + 1)}}\frac{1}{\binom{2j}j}\frac{{F_{2j} }}{{3^{j + 1} }}}  = \frac{2}{{\sqrt 5 }}\arctan \left( {\frac{{\sqrt 5 }}{2}} \right),\quad\sum_{j = 0}^\infty  {\frac{{2^{2j + 1} }}{{(2j + 1)}}\frac{1}{\binom{2j}j}\frac{{L_{2j} }}{{3^{j + 1} }}}  = \pi.
\]
The above identities are only special cases of more general identities to be derived in section~\ref{sec.results}.

The reader will observe that the infinite series stated above contain the inverse of the central binomial coefficients. We will also establish identities involving reciprocals of non-central binomial coefficients. Specifically we will evaluate the following sums:
\[
\sum_{j = m}^\infty  {\frac{{F_{r(n + j)} }}{{L_r^{n + j} \binom{n + j}j}}} ,\quad\sum_{j = m}^\infty  {\frac{{L_{r(n + j)} }}{{L_r^{n + j} \binom{n + j}j}}},
\]
and similar sums, for non-negative integers $m$ and $n$ and even integer $r$.

The key ingredients for obtaining our infinite series results are the following identities of Lehmer~\cite[identities (9), (13) and (15)]{lehmer85}:
\begin{equation}\label{eq.n6xwmlq}
\sum_{j = 1}^\infty  {\frac{{2^{2j}z^{2j} }}{{j\binom{2j}j}}}  = \frac{{2z}}{{\sqrt {1 - z^2 } }}\arcsin z,
\end{equation}

\begin{equation}
\sum_{j = 1}^\infty  {\frac{{2^{2j}z^{2j} }}{{j^2 \binom{2j}j}}}  = 2(\arcsin z)^2, 
\end{equation}

\begin{equation}\label{eq.r5clsui}
\sum_{j = 1}^\infty  {\frac{{2^{2j}z^{2j} }}{\binom{2j}j}}  = \frac{{z^2 }}{{1 - z^2 }} + \frac{{z}}{{(1 - z^2 )^{3/2} }}\arcsin z,
\end{equation}
each of which is valid for $|z|<1$; as well as the classical Euler series for the inverse tangent (Castellanos \cite[Equation (33)]{castellanos86}):
\begin{equation}\label{eq.a64cexn}
\sum_{j = 0}^\infty  {\frac{{2^{2j} }}{{2j + 1}}\frac{1}{\binom{2j}j}\left( {\frac{{z^2 }}{{1 + z^2 }}} \right)^{j + 1} }  = z\arctan z,
\end{equation}
valid for all real $z$.

For the derivation of identities involving inverses of non-central binomial coefficients we require the following identity established by Sury et.~al.~\cite[Theorem 3.6]{sury04}:
\begin{equation}\label{eq.icn6htq}
\begin{split}
\sum_{j = m}^\infty  {\frac{{z^{n + j} }}{\binom {n + j}j}}  &= n\sum_{j = 1}^{n - 1} {\binom {n - 1}j\frac{{(z - 1)^{n - j - 1} }}{{j\binom {m + j}j}}}  - n\sum_{j = 1}^m {\binom {m}j\frac{{(z - 1)^{n + j - 1} }}{{j\binom {n - 1 + j}j}}}\\
&\qquad + n(z - 1)^{n - 1} \sum_{j = m + 1}^{n - 1} {\frac{1}{j}}  + n(z - 1)^{n - 1} \log \left( {\frac{1}{{1 - z}}} \right),
\end{split}
\end{equation}
which holds for non-negative integers $m$ and $n$ and $|z|<1$.

The Fibonacci numbers, $F_j$, and the Lucas numbers, $L_j$, are defined, for \text{$j\in\mathbb Z$}, through the recurrence relations 
\begin{equation}\label{eq.s6z1bcx}
F_j=F_{j-1}+F_{j-2}, \text{($j\ge 2$)},\quad\text{$F_0=0$, $F_1=1$};
\end{equation}
and
\begin{equation}
L_j=L_{j-1}+L_{j-2}, \text{($j\ge 2$)},\quad\text{$L_0=2$, $L_1=1$};
\end{equation}
with
\begin{equation}
F_{-j}=(-1)^{j-1}F_j,\quad L_{-j}=(-1)^jL_j.
\end{equation}

Explicit formulas (Binet formulas) for the Fibonacci and Lucas numbers are
\begin{equation}
F_j  = \frac{{\alpha ^j  - \beta ^j }}{{\alpha  - \beta }},\quad L_j  = \alpha ^j  + \beta ^j,\quad j\in\mathbb Z.
\end{equation}
Koshy \cite{koshy} and Vajda \cite{vajda} have written excellent books dealing with Fibonacci and Lucas numbers.

In some cases the results readily extend to the more general second order sequence $(w_j(a,b;p,q))$, the Horadam sequence~\cite{horadam65}, defined recursively for all non-negative integers $j$ by
\begin{equation*}\label{Def-Horadam}
w_0  = a,\,\, w_1  = b, \quad w_j  = pw_{j - 1}  - qw_{j - 2},\quad j\ge 2,
\end{equation*}
where $a$, $b$, $p$ and $q$ are arbitrary complex numbers with non-zero $p$ and $q$.

The sequences $u_j(p,q)=w_j(0,1;p,q)$ and $v_j(p,q)=w_n(2,p;p,q)$ are called Lucas sequences of the first kind and of the second kind, respectively.

\par The Binet formulas for sequences $u_n$, $v_n$ and $w_n$ in the non-degenerate case, $p^2 - 4q > 0$, are
\[
u_n=\frac{\tau^n-\sigma^n}{\tau-\sigma},\qquad v_n=\tau^n+\sigma^n, \qquad
w_n = A\tau ^n  + B\sigma ^n\,,
\]
with
\[
A=\frac{{b - a\sigma}}{{\tau  - \sigma }},\quad B=\frac{{a\tau  - b}}{{\tau  - \sigma }},
\]
where 
\[\tau=\frac{p+\sqrt{p^2-4q}}2,\quad\sigma=\frac{p-\sqrt{p^2-4q}}2,\]
are the distinct zeros of the characteristic polynomial $x^2-px+q$ of the Horadam sequence; so that $\tau\sigma=q$ and $\tau + \sigma=p$.
\section{Main results}\label{sec.results}
\subsection{Finite sums}
\begin{theorem}
If $n$ is a non-negative integer and $s$ is any integer, then,
\begin{gather}
\sum_{j = 0}^n {\frac{{F_{j + n + s} }}{\binom nj}}  = (n + 1)\sum_{j = 0}^n {\frac{{F_{j + s - 2}  + F_{2j + s - 1} }}{{j + 1}}},\label{eq.avcrqb3}\\
\nonumber\\
\sum_{j = 0}^n {\frac{{L_{j + n + s} }}{\binom nj}}  = (n + 1)\sum_{j = 0}^n {\frac{{L_{j + s - 2}  + L_{2j + s - 1} }}{{j + 1}}}.\label{eq.eco8bvu}
\end{gather}
\end{theorem}
\begin{proof}
Set $z=\alpha$, $z=\beta$, in turn in identity~\eqref{eq.ibqxgb5} to obtain
\begin{gather}
\sum_{j = 0}^n {\frac{{\alpha ^{j + n + s} }}{\binom nj}}  = (n + 1)\sum_{j = 0}^n {\frac{{\alpha ^{j + s - 2}  + \alpha ^{2j + s - 1} }}{{j + 1}}},\nonumber\\
\nonumber\\
\sum_{j = 0}^n {\frac{{\beta ^{j + n + s} }}{\binom nj}}  = (n + 1)\sum_{j = 0}^n {\frac{{\beta ^{j + s - 2}  + \beta ^{2j + s - 1} }}{{j + 1}}},\nonumber
\end{gather}
where $s$ is an arbitrary integer; and hence identities \eqref{eq.avcrqb3} and \eqref{eq.eco8bvu}.
\end{proof}
\begin{theorem}
If $n$ is a non-negative integer and $s$ is any integer, then,
\begin{equation}\label{eq.ets6dy6}
\begin{split}
\sum_{j = 0}^{2n} {\frac{{F_{j + s} }}{{2^{j + s} \binom {2n}j}}}  &= \frac{{2n + 1}}{{5^{n + 1} }}\left( {F_{s - 1} \sum_{j = 1}^n {\frac{{5^j }}{{2j}}}  + L_{s - 1} \sum_{j = 0}^n {\frac{{5^j }}{{2j + 1}}} } \right)\\
&\qquad+ \frac{{2n + 1}}{{5^{n + 1} }}\left( {\sum_{j = 1}^n {\frac{{5^j F_{2j + s - 1} }}{{2^{2j} 2j}}}  + \sum_{j = 0}^n {\frac{{5^j L_{2j + s} }}{{2^{2j + 1} (2j + 1)}}} } \right),
\end{split}
\end{equation}

\begin{equation}\label{eq.yclk0eb}
\begin{split}
\sum_{j = 0}^{2n} {\frac{{L_{j + s} }}{{2^{j + s} \binom {2n}j}}}  &= \frac{{2n + 1}}{{5^n }}\left( {F_{s - 1} \sum_{j = 0}^n {\frac{{5^j }}{{2j + 1}}}  + L_{s - 1} \sum_{j = 1}^n {\frac{{5^{j - 1} }}{{2j}}} } \right)\\
&\qquad+ \frac{{2n + 1}}{{5^n }}\left( {\sum_{j = 0}^n {\frac{{5^j F_{2j + s} }}{{2^{2j + 1} (2j + 1)}}}  + \sum_{j = 1}^n {\frac{{5^{j - 1} L_{2j + s - 1} }}{{2^{2j} 2j}}} } \right).
\end{split}
\end{equation}
\end{theorem}
\begin{proof}
Write $2n$ for $n$ and set $z=\alpha/2$, $z=\beta/2$, in turn, in \eqref{eq.ibqxgb5} to obtain
\begin{equation}\label{eq.ml2hnke}
5^n\sqrt 5 \sum_{j = 0}^{2n} {\frac{{\alpha ^{j + s} }}{{2^{j + 1} \binom {2n}j}}}  = (2n + 1)\sum_{j = 0}^{2n} {\frac{{\alpha^{s - 1}(\sqrt 5 )^j }}{{(j + 1)}}}  + (2n + 1)\sum_{j = 0}^{2n} {\frac{{\alpha ^{j + s} (\sqrt 5 )^j }}{{2^{j + 1} (j + 1)}}},
\end{equation}

\begin{equation}\label{eq.o7kwao5}
-5^n\sqrt 5 \sum_{j = 0}^{2n} {\frac{{\beta ^{j + s} }}{{2^{j + 1} \binom {2n}j}}}  = (2n + 1)\sum_{j = 0}^{2n} {\frac{{\beta^{s - 1}(-\sqrt 5 )^j }}{{(j + 1)}}}  + (2n + 1)\sum_{j = 0}^{2n} {\frac{{\beta ^{j + s} (-\sqrt 5 )^j }}{{2^{j + 1} (j + 1)}}},
\end{equation}
where $s$ is an arbitrary integer.

Rewrite each sum on the right hand side of identities \eqref{eq.ml2hnke} and \eqref{eq.o7kwao5} according to the summation identity:
\[
\sum_{k = 0}^{2n} {f_k }  = \sum_{k = 0}^n {f_{2k} }  + \sum_{k = 1}^n {f_{2k - 1} }, 
\]
where $(f_k)$ is any real sequence.
 
Addition of the resulting equations gives \eqref{eq.ets6dy6} while their difference yields identity \eqref{eq.yclk0eb}.
\end{proof}
\begin{theorem}
Let $s$ and $r$ be any integers. Let $n$ be a non-negative integer. Then,
\begin{gather}
\sum_{j = 0}^n {\frac{{( - 1)^{rj} F_{2rj + s} }}{\binom nj}}  = \frac{{(n + 1)F_{rn + s} }}{{L_r^{n + 1} }}\sum_{j = 0}^n {\frac{{( - 1)^{rj} L_r^{j} L_{r(j + 1)} }}{{j + 1}}},\tag{\ref{eq.t8a4tnh}}\\
\nonumber\\
\sum_{j = 0}^n {\frac{{( - 1)^{rj} L_{2rj + s} }}{\binom nj}}  = \frac{{(n + 1)L_{rn + s} }}{{L_r^{n + 1} }}\sum_{j = 0}^n {\frac{{( - 1)^{rj} L_r^{j} L_{r(j + 1)} }}{{j + 1}}}\tag{\ref{eq.ne62ilo}}.
\end{gather}
\end{theorem}
\begin{proof}
In identity~\eqref{eq.ibqxgb5}, set $z=\alpha^r/\beta^r$ and multiply the resulting identity by $\alpha^s$ where $s$ is any integer. This gives
\begin{equation}\label{eq.t5zl01i}
\sum_{j = 0}^n {\frac{{\alpha ^{rj + s} }}{{\beta ^{rj} \binom nj}}}  = (n + 1)\frac{{\alpha ^{rn + s} \beta ^r }}{{(\alpha^r +\beta^r)^{n + 1} }}\sum_{j = 0}^n {\frac{(\alpha ^{r(j + 1)}  + \beta ^{r(j + 1)} )}{{\beta ^{r(j + 1)} (j + 1)}}\frac{{(\alpha^r +\beta^r)^j }}{{\alpha ^{rj} }}}.
\end{equation}
Similarly, we have
\begin{equation}\label{eq.b30d2ga}
\sum_{j = 0}^n {\frac{{\beta ^{rj + s} }}{{\alpha ^{rj} \binom nj}}}  = (n + 1)\frac{{\beta ^{rn + s} \alpha ^r }}{{(\alpha^r +\beta^r)^{n + 1} }}\sum_{j = 0}^n {\frac{(\alpha ^{r(j + 1)}  + \beta ^{r(j + 1)} )}{{\alpha ^{r(j + 1)} (j + 1)}}\frac{{(\alpha^r +\beta^r)^j }}{{\beta ^{rj} }}}.
\end{equation}
Combining \eqref{eq.t5zl01i} and \eqref{eq.b30d2ga}, we have
\begin{equation}\label{eq.uy9l2wg}
\sum_{j = 0}^n {\frac{{(f\alpha ^{2rj + s}  + g\beta ^{2rj + s} )}}{{\binom nj(\alpha \beta )^{rj} }}}  = \frac{{(n + 1)}}{{(\alpha^r +\beta^r)^{n + 1} }}(f\alpha ^{rn + s}  + g\beta ^{rn + s} )\sum_{j = 0}^n {\frac{{(\alpha ^{r(j + 1)}  + \beta ^{r(j + 1)} )(\alpha^r +\beta^r)^j }}{{(\alpha \beta )^{rj} (j + 1)}}},
\end{equation}
where $f$ and $g$ are arbitrary. Setting $(f,g)=(1,-1)$ and using the Binet formula gives identity~\eqref{eq.t8a4tnh}, while setting $(f,g)=(1,1)$ gives identity~\eqref{eq.ne62ilo}.
\end{proof}
\begin{remark}
Using $z=\tau^r/\sigma^r$ in \eqref{eq.ibqxgb5} gives, analogous to \eqref{eq.uy9l2wg},
\begin{equation}\label{eq.ygozmuk}
\sum_{j = 0}^n {\frac{{(f\tau ^{2rj + s}  + g\sigma ^{2rj + s} )}}{{\binom nj(\tau \sigma )^{rj} }}}  = \frac{{(n + 1)}}{{(\tau^r +\sigma^r)^{n + 1} }}(f\tau ^{rn + s}  + g\sigma ^{rn + s} )\sum_{j = 0}^n {\frac{{(\tau ^{r(j + 1)}  + \sigma ^{r(j + 1)} )(\tau^r +\sigma^r)^j }}{{(\tau \sigma )^{rj} (j + 1)}}},
\end{equation}
\end{remark}
from which, with $(f,g)=(A,B)$, we get
\begin{equation}
\sum_{j = 0}^n {\frac{{ w_{2rj + s} }}{q^{rj}\binom nj}}  = \frac{{(n + 1)w_{rn + s} }}{{v_r^{n + 1} }}\sum_{j = 0}^n {\frac{{v_r^{j} v_{r(j + 1)} }}{q^{rj}{(j + 1)}}},
\end{equation}
for non-negative integer $n$ and integers $r$ and $s$.
\begin{theorem}
If $n$ is a non-negative integer and $s$ is any integer, then,
\begin{gather}
\sum_{j = 0}^n {\frac{{F_{3j + s - n} }}{\binom nj}}  = \frac{{(n + 1)}}{{2^{n + 1} }}\sum_{j = 0}^n {\frac{{2^j }}{{j + 1}}(F_{s + 2j + 1}  + F_{s - j - 2} )},\label{eq.d4zydzo}\\
\nonumber\\
\sum_{j = 0}^n {\frac{{L_{3j + s - n} }}{\binom nj}}  = \frac{{(n + 1)}}{{2^{n + 1} }}\sum_{j = 0}^n {\frac{{2^j }}{{j + 1}}(L_{s + 2j + 1}  + L_{s - j - 2} )}\label{eq.ktz4pns},\\
\nonumber\\
\sum_{j = 0}^n {\frac{{( - 1)^j F_{3j + s - 2n} }}{\binom nj}}  = \frac{{n + 1}}{{2^{n + 1} }}\sum_{j = 0}^n {\frac{{2^j ( - 1)^j }}{{j + 1}}(F_{s + j + 2}  + ( - 1)^{j - 1} F_{s - 2j - 1} )},\label{eq.kzn4gwj}\\
\nonumber\\
\sum_{j = 0}^n {\frac{{( - 1)^j L_{3j + s - 2n} }}{\binom nj}}  = \frac{{n + 1}}{{2^{n + 1} }}\sum_{j = 0}^n {\frac{{2^j ( - 1)^j }}{{j + 1}}(L_{s + j + 2}  + ( - 1)^{j - 1} L_{s - 2j - 1} )}.\label{eq.r3o55f1}
\end{gather}

\end{theorem}
\begin{proof}
Setting $z=\alpha^3$ in \eqref{eq.ibqxgb5} gives
\begin{gather}
\sum_{j = 0}^n {\frac{{\alpha ^{3j + s - n} }}{\binom nj}}  = \frac{{n + 1}}{{2^{n + 1} }}\sum_{j = 0}^n {\frac{{(\alpha ^{s - j - 2}  + \alpha ^{s + 2j + 1} )2^j }}{{j + 1}}},\label{eq.l643u70}\\
\nonumber\\
\sum_{j = 0}^n {\frac{{\beta ^{3j + s - n} }}{\binom nj}}  = \frac{{n + 1}}{{2^{n + 1} }}\sum_{j = 0}^n {\frac{{(\beta ^{s - j - 2}  + \beta ^{s + 2j + 1} )2^j }}{{j + 1}}}.\label{eq.ie79koi}
\end{gather}
Subtraction of \eqref{eq.ie79koi} from \eqref{eq.l643u70} gives \eqref{eq.d4zydzo}, while their addition yields \eqref{eq.ktz4pns}.
The proof of \eqref{eq.kzn4gwj} and \eqref{eq.r3o55f1} proceeds in the same manner; use $z=-\alpha^3$ in \eqref{eq.ibqxgb5}.
\end{proof}
\subsection{Infinite sums}
Our first main result on infinite series involving Fibonacci and Lucas numbers and reciprocal binomial coefficients is given in Theorem \ref{thm.lgnqry0}. First we state a couple of Lemmas, needed for its proof.
\begin{lemma}[See also {\cite[p.271, identities (20)--(22)]{srivastava12}}]\label{lem.mg2hc05}
We have
\begin{gather}
\arccos (\alpha /2) = \frac{\pi }{5},\quad\arccos ( - \beta /2) = \frac{{2\pi }}{5},\label{eq.bit0itz}\\
\nonumber\\
\arcsin (\alpha /2) = \frac{{3\pi }}{{10}},\quad\arcsin ( - \beta /2) = \frac{\pi }{{10}},\label{eq.vi9k0zi}\\
\nonumber\\
\cot (2\pi /5) =  - \beta ^3 \cot (\pi /5) =  - \beta ^3 \sqrt {\frac{{\alpha ^3 }}{{\sqrt 5 }}}\label{eq.vn4pi2d} .
\end{gather}
\end{lemma}
\begin{proof}
Since $3\pi/5=\pi - 2\pi/5$, it follows that $\sin(3\pi/5)=\sin(2\pi/5)$. By the multiple angle formulas, this means that
\[
3\sin(\pi/5) - 4\sin^3{(\pi/5)}=2\sin{(\pi/5)}\cos(\pi/5),
\]
which reduces to a quadratic equation in $\cos (\pi /5)$:
\[
4\cos ^2 (\pi /5) - 1 = 2\cos (\pi /5),
\]
from which we find
\[
\cos (\pi /5) = \frac{{1 + \sqrt 5 }}{4} = \frac{\alpha }{2},
\]
since the angle $\pi/5$ radians is in the first quadrant.

Now,
\[
\cos (2\pi /5) = 2\cos ^2 (\pi /5) - 1 = (\alpha ^2  - 2)/2 =  - \beta /2.
\]
Since $\arcsin x=\pi/2 - \arccos x$, the identities stated in \eqref{eq.bit0itz} imply those given in \eqref{eq.vi9k0zi}.

We have
\[
\begin{split}
\cot ^2 (\pi /5) &= \frac{{\cos ^2 (\pi /5)}}{{1 - \cos ^2 (\pi /5)}} = \frac{{\alpha ^2 /4}}{{1 - \alpha ^2 /4}}\\
&= \frac{{\alpha ^2 }}{{(2 - \alpha )(2 + \alpha )}} = \frac{{\alpha ^2 }}{{\beta ^2 \alpha \sqrt 5 }} = \frac{{\alpha ^3 }}{{\sqrt 5 }}
\end{split}
\]
and
\[
\begin{split}
\cot (2\pi /5) &= \frac{{2\cos ^2 (\pi /5) - 1}}{{2\cos ^2 (\pi /5)}}\cot (\pi /5) = \frac{{(\alpha ^2  - 2)}}{{\alpha ^2 }}\cot (\pi /5)\\
&= \frac{{ - \beta }}{{\alpha ^2 }}\cot (\pi /5) =  - \beta ^3 \cot (\pi /5).
\end{split}
\]
\end{proof}
\begin{lemma}\label{lem.oakhmtx}
For any integer $r$,
\begin{gather}
3\alpha ^r  - \beta ^{r + 3}  = L_{r + 1} \sqrt 5  - L_{r - 1},\\
3\alpha ^r  + \beta ^{r + 3}  = \sqrt 5 (F_{r + 1} \sqrt 5  - F_{r - 1} ).\label{eq.rch7jzs}
\end{gather}
\end{lemma}
\begin{proof}
We have
\[
\begin{split}
3\alpha ^r  - \beta ^{r + 3}  &= 2\alpha ^r  + \alpha ^r  - \beta ^{r + 3}\\ 
 &= \alpha ^{r + 2}  - \alpha ^{r - 1}  + \alpha ^r  - \beta ^{r - 1}  - \beta ^r  - \beta ^{r + 2}\\ 
 &= \alpha ^{r + 2}  + \alpha ^r  - \beta ^r  - \beta ^{r + 2}  - \alpha ^{r - 1}  - \beta ^{r - 1}\\ 
 &= (\alpha ^{r + 1}  + \beta ^{r + 1} )(\alpha  - \beta ) - (\alpha ^{r - 1}  + \beta ^{r - 1} )\\
 &= L_{r + 1} \sqrt 5  - L_{r - 1}.
\end{split}
\]
The proof of \eqref{eq.rch7jzs} is similar.
\end{proof}
\begin{theorem}\label{thm.lgnqry0}
If $s$ is an integer, then,
\begin{gather}
\sum_{j = 1}^\infty  {\frac{{F_{2j + s} }}{{j\binom{2j}j}}}  = (F_{s + 1} \sqrt 5  - F_{s - 1} )\frac{\pi }{5}\sqrt {\frac{{\alpha ^3 }}{{\sqrt 5 }}},\\
\nonumber\\ 
\sum_{j = 1}^\infty  {\frac{{L_{2j + s} }}{{j\binom{2j}j}}}  = (L_{s + 1} \sqrt 5  - L_{s - 1} )\frac{\pi }{5}\sqrt {\frac{{\alpha ^3 }}{{\sqrt 5 }}}. 
\end{gather}
\end{theorem}
\begin{proof}
It is convenient to write identity~\eqref{eq.n6xwmlq} as
\begin{equation}\label{eq.c7c4lo4}
\sum_{j = 1}^\infty  {\frac{{2^{2j} z^{2j} }}{{j\binom{2j}j}}}  = 2\arcsin z\cot (\arccos z).
\end{equation}
Set $z=\alpha/2$, $z=-\beta/2$, in turn, in \eqref{eq.c7c4lo4} and use the identities in Lemma \ref{lem.mg2hc05} to obtain
\begin{gather}
\sum_{j = 1}^\infty  {\frac{{\alpha ^{2j + s} }}{{j\binom{2j}j}}}  = \frac{{3\pi }}{5}\alpha ^s \cot \left( {\frac{\pi }{5}} \right),\\
\nonumber\\
\sum_{j = 1}^\infty  {\frac{{\beta ^{2j + s} }}{{j\binom{2j}j}}}  = \frac{\pi }{5}\beta ^s \cot \left( {\frac{{2\pi }}{5}} \right)=-\frac{\pi }{5}\beta ^{s + 3} \cot \left( {\frac{{\pi }}{5}} \right),
\end{gather}
where $s$ is an arbitrary integer.

Thus,
\[
\begin{split}
\sum_{j = 1}^\infty  {\frac{{(\alpha ^{2j + s}  - \beta ^{2j + s} )}}{{j\binom{2j}j}}}  &= \frac{\pi }{5}(3\alpha ^s  + \beta ^{s + 3} )\cot \left( {\frac{\pi }{5}} \right),\\
&\\
\sum_{j = 1}^\infty  {\frac{{(\alpha ^{2j + s}  + \beta ^{2j + s} )}}{{j\binom{2j}j}}}  &= \frac{\pi }{5}(3\alpha ^s  - \beta ^{s + 3} )\cot \left( {\frac{\pi }{5}} \right);
\end{split}
\]
from which, invoking Lemma \ref{lem.oakhmtx}, the identities stated in the theorem follow.
\end{proof}
\begin{remark}
We see that
\begin{gather}
\sum_{j = 1}^\infty  {\frac{{F_{2j + s} }}{{j\binom{2j}j}}}  = (F_{s + 1} \sqrt 5  - F_{s - 1} )\sum_{j = 1}^\infty  {\frac{{F_{2j - 1} }}{{j\binom{2j}j}}},\\
\nonumber\\
\sum_{j = 1}^\infty  {\frac{{L_{2j + s} }}{{j\binom{2j}j}}}  = \frac{{5\beta F_s  + 2L_{s + 1} }}{2}\sum_{j = 1}^\infty  {\frac{{L_{2j} }}{{j\binom{2j}j}}}.
\end{gather}
\end{remark}
\begin{example}
We have
\begin{gather}
\sum_{j = 1}^\infty  {\frac{{F_{2j - 3} }}{{j\binom{2j}j}}}  = \frac{{2\pi }}{5}\sqrt {\frac{{ - \beta }}{{\sqrt 5 }}} ,\quad\sum_{j = 1}^\infty  {\frac{{F_{2j - 1} }}{{j\binom{2j}j}}}  = \frac{\pi }{5}\sqrt {\frac{{\alpha ^3 }}{{\sqrt 5 }}} ,\\
\nonumber\\
\sum_{j = 1}^\infty  {\frac{{F_{2j} }}{{j\binom{2j}j}}}  = \frac{{2\pi }}{5}\sqrt {\frac{\alpha }{{\sqrt 5 }}} ,\quad\sum_{j = 1}^\infty  {\frac{{L_{2j} }}{{j\binom{2j}j}}}  = \frac{{2\pi }}{5}\sqrt {\frac{{\alpha ^5 }}{{\sqrt 5 }}}.
\end{gather}
\end{example}
\begin{lemma}[Borwein and Chamberland {\cite[Identity (1.1)]{borwein07}}]\label{lem.ips42od}
For $|z|<2$ and $m$ a positive integer,
\[
\sum_{j = 1}^\infty  {\frac{{H_m (j)}}{{j^2 \binom{2j}j}}z^{2j} }  = \frac1{2m!}\left( {\arcsin \left( {\frac{z}{2}} \right)} \right)^{2m},
\]
where $H_1(j)=1/4$ and
\[
H_{m + 1} (j) = \frac{1}{4}\sum_{n_1  = 1}^{j - 1} {\frac{1}{{(2n_1 )^2 }}\sum_{n_2  = 1}^{n_1  - 1} {\frac{1}{{(2n_2 )^2 }} \cdots \sum_{n_m  = 1}^{n_{m - 1}  - 1} {\frac{1}{{(2n_m )^2 }}} } }.
\]

\end{lemma}
\begin{theorem}
If $s$ is any integer and $m$ is a positive integer, then,
\begin{gather}
\sum_{j = 1}^\infty  {\frac{{H_m (j)}}{{j^2 \binom{2j}j}}F_{2j + s} }  = \frac{1}{{(2m)!}}\left( {\frac{\pi }{{10}}} \right)^{2m} \frac{1}{{\sqrt 5 }}((3^{2m}  + 1)\alpha ^s  - L_s ),\label{eq.y21ce2p}\\
\nonumber\\
\sum_{j = 1}^\infty  {\frac{{H_m (j)}}{{j^2 \binom{2j}j}}L_{2j + s} }  = \frac{1}{{(2m)!}}\left( {\frac{\pi }{{10}}} \right)^{2m} ((3^{2m}  - 1)\alpha ^s  + L_s ).\label{eq.rk5juo7}
\end{gather}
\end{theorem}
\begin{proof}
Set $z=\alpha$ in the identity given in Lemma~\ref{lem.ips42od} and multiplying through by $\alpha^s$, where $s$ is an arbitrary integer. This gives
\[
\sum_{j = 0}^\infty  {\frac{{H_m (j)}}{{j^2 \binom{2j}j}}\alpha ^{2j + s} }  = \frac{{\alpha ^s }}{{2m!}}\left( {\frac{{3\pi }}{{10}}} \right)^{2m}.
\]
Similarly, $z=-\beta$ gives
\[
\sum_{j = 0}^\infty  {\frac{{H_m (j)}}{{j^2 \binom{2j}j}}\beta ^{2j + s} }  = \frac{{\beta ^s }}{{2m!}}\left( {\frac{{\pi }}{{10}}} \right)^{2m}.
\]
Identities \eqref{eq.y21ce2p} and \eqref{eq.rk5juo7} follow from the addition and subtraction of the above identities.
\end{proof}
In particular, we have
\begin{gather}
\sum_{j = 1}^\infty  {\frac{{F_{2j + s} }}{{j^2 \binom{2j}j}}}  = \frac{{\pi ^2 }}{{5\sqrt 5 }}\left( {\alpha ^s  - \frac{{L_s }}{{10}}} \right),\\
\nonumber\\
\sum_{j = 1}^\infty  {\frac{{L_{2j + s} }}{{j^2 \binom{2j}j}} }  = \frac{{2\pi ^2 }}{{25}}\left( {2\alpha ^s  + \frac{{L_s }}{4}} \right),\\
\nonumber\\
\sum_{j = 1}^\infty  {\left\{ {\sum_{s = 1}^{j - 1} {\frac{1}{{s^2 }}} } \right\}\frac{{F_{2j + s} }}{{j^2 \binom{2j}j}}}  = \left( {\frac{\pi }{{10}}} \right)^4 \frac{2}{{3\sqrt 5 }}(82\alpha ^s  - L_s ),\\
\nonumber\\
\sum_{j = 1}^\infty  {\left\{ {\sum_{s = 1}^{j - 1} {\frac{1}{{s^2 }}} } \right\}\frac{{L_{2j + s} }}{{j^2 \binom{2j}j}}}  = \left( {\frac{\pi }{{10}}} \right)^4 \frac{2}{3}(80\alpha ^s  + L_s ).
\end{gather}
\begin{example}
We have
\begin{equation}
\sum_{j = 1}^\infty  {\frac{{F_{2j} }}{{j^2 \binom{2j}j}}}  = \frac{{4\pi ^2 \sqrt 5 }}{{125}},\quad\sum_{j = 1}^\infty  {\frac{{L_{2j} }}{{j^2 \binom{2j}j}}}  = \frac{{\pi ^2 }}{5},
\end{equation}

\begin{equation}\label{eq.eymtl1t}
\sum_{j = 1}^\infty  {\frac{{L_{2j + 3} }}{{j^2 \binom{2j}j}}}  = \frac{{2\pi ^2}}{25}\alpha^3\sqrt5,\quad\sum_{j = 1}^\infty  {\frac{{L_{2j - 3} }}{{j^2 \binom{2j}j}}}  = \frac{{2\pi ^2}}{25}\beta^3\sqrt5,
\end{equation}

\begin{equation}
\sum_{j = 1}^\infty  {\left\{ {\sum_{s = 1}^{j - 1} {\frac{1}{{s^2 }}} } \right\}\frac{{F_{2j} }}{{j^2 \binom{2j}j}}}  = \frac{{27\pi ^4 \sqrt 5 }}{{25000}},\quad\sum_{j = 1}^\infty  {\left\{ {\sum_{s = 1}^{j - 1} {\frac{1}{{s^2 }}} } \right\}\frac{{L_{2j} }}{{j^2 \binom{2j}j}}}  = \frac{{41\pi ^4 }}{{4100}}.
\end{equation}
\end{example}
Note that in deriving \eqref{eq.eymtl1t}, we used $2\beta^3 + 1=-\beta^3\sqrt 5$ and $2\alpha^3 + 1=\alpha^3\sqrt 5$.
\begin{lemma}\label{lem.cpbvwfe}
If $r$ is an integer, then,
\begin{gather}
\alpha^r - \beta^{r + 6} = -\beta^3L_{r + 3},\nonumber\\
\alpha^r + \beta^{r + 6} = -\beta^3F_{r + 3}\sqrt 5.\nonumber
\end{gather}
\end{lemma}
\begin{theorem}
If $s$ is an integer, then,
\begin{gather}
\sum_{j = 1}^\infty  {\frac{{F_{2j + s} }}{\binom{2j}j}}  = \frac{{L_{s + 3} }}{5} + (L_{s + 2} \sqrt 5  - L_s )\frac{{2\pi }}{{25}}\sqrt {\frac{{\alpha ^3 }}{{\sqrt 5 }}},\\ 
\nonumber\\
\sum_{j = 1}^\infty  {\frac{{L_{2j + s} }}{\binom{2j}j}}  = F_{s + 3}  + (F_{s + 2} \sqrt 5  - F_s )\frac{{2\pi }}{{5}}\sqrt {\frac{{\alpha ^3 }}{{\sqrt 5 }}}. 
\end{gather}
\end{theorem}
\begin{proof}
Write identity~\eqref{eq.r5clsui} as
\[
\sum_{j = 1}^\infty  {\frac{{2^{2j} z^{2j} }}{\binom{2j}j}}  = \cot ^2 (\arccos z) + \frac{1}{{z^2 }}\arcsin z\cot ^3 (\arccos z).
\]
Set $z=\alpha/2$, $z=-\beta/2$ in this identity to obtain
\begin{gather}
\sum_{j = 1}^\infty  {\frac{{2^{2j} \alpha ^{2j + s} }}{\binom{2j}j}}  = \alpha ^s \cot ^2 (\pi /5) + \frac{{6\pi }}{5}\alpha ^{s - 2} \cot ^3 (\pi /5),\nonumber
\nonumber\\
\sum_{j = 1}^\infty  {\frac{{2^{2j} \beta ^{2j + s} }}{\binom{2j}j}}  = \beta ^s \cot ^2 (2\pi /5) + \frac{{2\pi }}{5}\beta ^{s - 2} \cot ^3 (2\pi /5),\nonumber
\end{gather}
where $s$ is an arbitrary integer.

Using the Binet formulas and identity \eqref{eq.vn4pi2d} we find
\[
\sqrt 5 \sum_{j = 1}^\infty  {\frac{{F_{2j + s} }}{\binom{2j}j}}  = (\alpha ^s  - \beta ^{s + 6} )\cot ^2 \left( {\frac{\pi }{5}} \right) + \frac{{2\pi }}{5}(3\alpha ^{s - 2}  + \beta ^{s + 7} )\cot ^3 \left( {\frac{\pi }{5}} \right),
\]

\[
\sum_{j = 1}^\infty  {\frac{{L_{2j + s} }}{\binom{2j}j}}  = (\alpha ^s  + \beta ^{s + 6} )\cot ^2 \left( {\frac{\pi }{5}} \right) + \frac{{2\pi }}{5}(3\alpha ^{s - 2}  - \beta ^{s + 7} )\cot ^3 \left( {\frac{\pi }{5}} \right).
\]
Since 
\[
\cot ^2 (\pi /5) = \frac{{\alpha ^3 }}{{\sqrt 5 }},\quad\cot ^3 (\pi /5) = \frac{{\alpha ^3 }}{{\sqrt 5 }}\sqrt {\frac{{\alpha ^3 }}{{\sqrt 5 }}}, 
\]
we have
\begin{gather}
\sum_{j = 1}^\infty  {\frac{{F_{2j + s} }}{\binom{2j}j}}  = (\alpha ^s  - \beta ^{s + 6} )\frac{{\alpha ^3 }}{{5}} + \frac{{2\pi }}{{25}}(3\alpha ^{s + 1}  - \beta ^{s + 4} )\sqrt {\frac{{\alpha ^3 }}{{\sqrt 5 }}}\nonumber\\ 
\nonumber\\
\sum_{j = 1}^\infty  {\frac{{L_{2j + s} }}{\binom{2j}j}}  = (\alpha ^s  + \beta ^{s + 6} )\frac{\alpha ^3}{\sqrt 5}  + \frac{{2\pi }}{5\sqrt 5}(3\alpha ^{s + 1}  + \beta ^{s + 4} )\sqrt {\frac{{\alpha ^3 }}{{\sqrt 5 }}}, 
\end{gather}
from which the stated results follow upon using Lemma \ref{lem.cpbvwfe} and Lemma \ref{lem.oakhmtx}.
\end{proof}
\begin{example}
We have
\begin{equation}
\sum_{j = 1}^\infty  {\frac{{F_{2j} }}{\binom{2j}j}}  = \frac{4}{5} + \frac{{2\pi }}{{25}}\frac{{(3\sqrt 5  - 2)}}{{\sqrt 5 }}\sqrt {\alpha ^3 \sqrt 5 },
\end{equation}

\begin{equation}
\sum_{j = 1}^\infty  {\frac{{L_{2j} }}{\binom{2j}j}}  = 2 + \frac{{2\pi }}{5}\sqrt {\alpha ^3 \sqrt 5 },
\end{equation}

\begin{equation}
\sum_{j = 1}^\infty  {\frac{{F_{2j - 1} }}{\binom{2j}j}}  = \frac{3}{5} + \frac{{4\pi }}{{25}}\sqrt {\frac{{\alpha ^5 }}{{\sqrt 5 }}},\quad\sum_{j = 1}^\infty  {\frac{{L_{2j - 1} }}{\binom{2j}j}}  = 1 + \frac{{4\pi }}{5}\sqrt {\frac{\alpha }{{\sqrt 5 }}},
\end{equation}

\begin{equation}
\sum_{j = 1}^\infty  {\frac{{L_{2j - 2} }}{\binom{2j}j}}  = 1 + \frac{{2\pi }}{5}\sqrt {\frac{{\alpha ^3 }}{{\sqrt 5 }}},\quad\sum_{j = 1}^\infty  {\frac{{L_{2j - 3} }}{\binom{2j}j}}  = \frac{{2\pi }}{5}\sqrt {\frac{{ - \beta ^3 }}{{\sqrt 5 }}}.
\end{equation}

\end{example}
\begin{lemma}\label{lem.sju45tu}
Let $r$ be any integer. Then $\alpha^{2r}=1/\beta^{2r}>0$ and
\[
\arctan (\alpha ^{2r} ) + \arctan (\beta ^{2r} ) = \frac{\pi }{2},
\]

\[
\begin{split}
\arctan (\alpha ^{2r} ) - \arctan (\beta ^{2r} ) &= \arctan \left( {\frac{{\alpha ^{2r}  - \beta ^{2r} }}{{1 + \alpha ^{2r} \beta ^{2r} }}} \right)\\
&= \arctan \left( {\frac{{F_{2r} \sqrt 5 }}{2}} \right).
\end{split}
\]
\end{lemma}
\begin{lemma}\label{lem.txhusoe}
For arbitrary $f$ and $g$ and any integer $s$,
\[\alpha^sf+\beta^sg=\frac{L_s}2(f + g) +\frac{F_s\sqrt 5}2(f-g).\]
\end{lemma}
\begin{theorem}
If $r$ and $s$ are integers, then,
\begin{gather}
\sum_{j = 0}^\infty  {\frac{{2^{2j + 1} }}{(2j + 1)}\frac{1}{\binom{2j}j}\frac{{F_{2rj + s} }}{{L_{2r}^{j + 1} }}}  = \frac{\pi }{2}F_s  + \frac{{L_s }}{{\sqrt 5 }}\arctan \left( {\frac{{F_{2r} \sqrt 5 }}{2}} \right),\label{eq.mbglv58}\\
\nonumber\\
\sum_{j = 0}^\infty  {\frac{{2^{2j + 1} }}{(2j + 1)}\frac{1}{\binom{2j}j}\frac{{L_{2rj + s} }}{{L_{2r}^{j + 1} }}}  = \frac{\pi }{2}L_s  + F_s \sqrt 5 \arctan \left( {\frac{{F_{2r} \sqrt 5 }}{2}} \right)\label{eq.wj0o96c}.
\end{gather}
\end{theorem}
\begin{proof}
In \eqref{eq.a64cexn}, set $z=\alpha^{2r}$, $z=\beta^{2r}$, in turn. This gives
\begin{gather}
\sum_{j = 0}^\infty  {\frac{{2^{2j} }}{{2j + 1}}\frac{1}{\binom{2j}j}\frac{{\alpha ^{2r(j + 1)} }}{{L_{2r}^{j + 1} }}}  = \alpha ^{2r} \arctan \alpha ^{2r}\nonumber,\\
\nonumber\\
\sum_{j = 0}^\infty  {\frac{{2^{2j} }}{{2j + 1}}\frac{1}{\binom{2j}j}\frac{{\beta ^{2r(j + 1)} }}{{L_{2r}^{j + 1} }}}  = \beta ^{2r} \arctan \beta ^{2r};\nonumber 
\end{gather}
so that
\begin{gather}
\sum_{j = 0}^\infty  {\frac{{2^{2j} }}{{2j + 1}}\frac{1}{\binom{2j}j}\frac{{\alpha ^{2rj + s} }}{{L_{2r}^{j + 1} }}}  = \alpha ^s \arctan \alpha ^{2r}\nonumber,\\
\nonumber 
\sum_{j = 0}^\infty  {\frac{{2^{2j} }}{{2j + 1}}\frac{1}{\binom{2j}j}\frac{{\beta ^{2rj + s} }}{{L_{2r}^{j + 1} }}}  = \beta ^s \arctan \beta ^{2r},\nonumber 
\end{gather}
where $s$ is an arbitrary integer.

Thus,
\begin{gather}
\sqrt 5\sum_{j = 0}^\infty  {\frac{{2^{2j} }}{{2j + 1}}\frac{1}{\binom{2j}j}\frac{{F_{2rj + s} }}{{L_{2r}^{j + 1} }}}  = \alpha ^s \arctan \alpha ^{2r} - \beta ^s \arctan \beta ^{2r}\nonumber,\\
\nonumber \\
\sum_{j = 0}^\infty  {\frac{{2^{2j} }}{{2j + 1}}\frac{1}{\binom{2j}j}\frac{{L_{2rj + s} }}{{L_{2r}^{j + 1} }}}  =\alpha ^s \arctan \alpha ^{2r} + \beta ^s \arctan \beta ^{2r},\nonumber 
\end{gather}
and employing Lemma \ref{lem.sju45tu} and Lemma \ref{lem.txhusoe}, the results stated in the theorem follow.
\end{proof}
In particular, we have
\begin{gather}
\sum_{j = 0}^\infty  {\frac{{2^{2j + 1} F_{2j + s} }}{{(2j + 1)\binom{2j}j3^{j + 1} }}}  = \frac{\pi }{2}F_s  + \frac{{L_s }}{{\sqrt 5 }}\arctan \left( {\frac{{\sqrt 5 }}{2}} \right),\\
\nonumber\\
\sum_{j = 0}^\infty  {\frac{{2^{2j + 1} L_{2j + s} }}{{(2j + 1)\binom{2j}j3^{j + 1} }}}  = \frac{\pi }{2}L_s  + F_s \sqrt 5\arctan \left( {\frac{{\sqrt 5 }}{2}} \right).
\end{gather}
\begin{example}
We have
\begin{equation}
\sum_{j = 0}^\infty  {\frac{{2^{j + 1} }}{(2j + 1)}\frac{1}{\binom{2j}j}}  = \pi,
\end{equation}

\begin{equation}
\sum_{j = 0}^\infty  {\frac{{2^{2j + 1} }}{{(2j + 1)}}\frac{1}{\binom{2j}j}\frac{{F_{2j} }}{{3^{j + 1} }}}  = \frac{2}{{\sqrt 5 }}\arctan \left( {\frac{{\sqrt 5 }}{2}} \right),\quad\sum_{j = 0}^\infty  {\frac{{2^{2j + 1} }}{{(2j + 1)}}\frac{1}{\binom{2j}j}\frac{{L_{2j} }}{{3^{j + 1} }}}  = \pi.
\end{equation}
\end{example}
\begin{theorem}\label{thm.eifjc82}
Let $r$ be an even integer, $n$ a positive integer and $m$ a non-negative integer. Then,
\begin{equation}
\begin{split}
\frac{{( - 1)^n }}{n}\sum_{j = m}^\infty  {\frac{{F_{r(n + j)} }}{{L_r^{n + j} \binom{n + j}j}}}  &= \sum_{j = 1}^{n - 1} {\binom{n - 1}j\frac{{( - 1)^j F_{r(n - j - 1)} }}{{jL_r^{n - j - 1} \binom{m + j}j}}}  - \sum_{j = 1}^m {\binom mj\frac{{( - 1)^j F_{r(n + j - 1)} }}{{jL_r^{n + j - 1} \binom{m + j - 1}j}}}\\ 
&\qquad+ \frac{{F_{r(n - 1)} }}{{L_r^{n - 1} }}\sum_{j = m + 1}^{n - 1} {\frac{1}{j}}  + \frac{1}{{L_r^{n - 1} }}\left( {F_{r(n - 1)} \log L_r  - \frac{{rL_{r(n - 1)} }}{{\sqrt 5 }}\log \alpha } \right),
\end{split}
\end{equation}

\begin{equation}
\begin{split}
\frac{{( - 1)^{n - 1} }}{n}\sum_{j = m}^\infty  {\frac{{L_{r(n + j)} }}{{L_r^{n + j} \binom{n + j}j}}}  &= \sum_{j = 1}^{n - 1} {\binom{n - 1}j\frac{{( - 1)^j L_{r(n - j - 1)} }}{{jL_r^{n - j - 1} \binom{m + j}j}}}  - \sum_{j = 1}^m {\binom mj\frac{{( - 1)^j L_{r(n + j - 1)} }}{{jL_r^{n + j - 1} \binom{m + j - 1}j}}}\\ 
&\qquad+ \frac{{L_{r(n - 1)} }}{{L_r^{n - 1} }}\sum_{j = m + 1}^{n - 1} {\frac{1}{j}}  + \frac{1}{{L_r^{n - 1} }}\left( {L_{r(n - 1)} \log L_r  - rF_{r(n - 1)}\sqrt 5 \log \alpha } \right).
\end{split}
\end{equation}
\end{theorem}
\begin{proof}
If $r$ is an even integer, then, obviously $\alpha^r/L_r<1$, $\beta^r/L_r<1$. Set $z=\alpha^r/L_r$, $z=\beta^r/L_r$, in turn, in \eqref{eq.icn6htq}. Addition and subtraction of the resulting identities yield the results stated in the theorem.
\end{proof}



\hrule

\bigskip

\bigskip
\noindent Concerned with sequences: A000032, A000045, A000129, A001045, A001582, A002450, A014551

\bigskip
\hrule
\bigskip

\vspace*{+.1in}
\noindent



\end{document}